\documentclass[preprint,12pt]{elsarticle}

\usepackage{graphics}
\usepackage{cancel}
\usepackage{tikz}
\usetikzlibrary{matrix,arrows}
\usepackage{amsmath}
\usepackage{amsthm}
\usepackage{amssymb}
\usepackage{amsrefs}
\usepackage{natbib}

\newcommand*\cocolon{%
        \nobreak
        \mskip6mu plus1mu
        \mathpunct{}%
        \nonscript
        \mkern-\thinmuskip
        {:}%
        \mskip2mu
        \relax
}

\newenvironment{newlist}
   {\begin{list}{}{\setlength{\labelsep}{0.25cm}
\setlength{\itemsep}{-0.1cm}
\setlength{\topsep}{0.1cm}
                   \setlength{\labelwidth}{0.7cm}
                      \setlength{\leftmargin}{0.8cm}}}  
   {\end{list}}

\newcommand{\op}[1]{\mathbb{#1}}
\newcommand{\ra}{\rightarrow}

\newcommand{\lan}{\langle}
\newcommand{\ran}{\rangle}
\newcommand{\e}{{\rm e}}
\newcommand{\mt}{\land}
\newcommand{\jn}{\lor}
\newcommand{\rd}{\slash}
\newcommand{\ld}{\backslash}
\newcommand{\eq}{\approx}

\newcommand{\m}[1]{{\bf {#1} }}

\newcommand{\si}{\ensuremath{\sigma}}

\DeclareMathOperator{\Con}{\mathrm{Con}}
\DeclareMathOperator{\KCon}{\mathrm{KCon}}
\newcommand{\cg}[1]{{\rm Cg}_{_{#1}}}

\newcommand{\cls}[1]{\mathcal{#1}}
\newcommand{\mdl}[1]{\models_{#1}}
\newcommand{\lgc}[1]{\ensuremath{{\sf #1}}}
\newcommand{\De}{\Delta}

\newcommand{\Si}{\Sigma}
\newcommand{\The}{\Theta}

\newcommand{\F}{\m{F}}
\newcommand{\Tm}{\m{Tm}}

\newcommand{\lang}{\mathcal{L}}
\newcommand{\x}{\overline{x}}
\newcommand{\y}{\overline{y}}
\newcommand{\z}{\overline{z}}

\newcommand{\csts}{\overline{c}}
\newcommand{\eps}{\ensuremath{\varepsilon}}
\newcommand{\N}{\mathbb{N}}
\newcommand{\V}{\mathcal{V}}
\newcommand{\W}{\mathcal{W}}
\newcommand{\K}{\mathcal{K}}
\newcommand{\C}{\mathcal{C}}
\newcommand{\Lat}{\lang\mathit{at}}
\newcommand{\FL}{\mathrm{FL}}

\theoremstyle{definition}
\newtheorem{Theorem}{Theorem}
\newtheorem{Proposition}[Theorem]{Proposition}
\newtheorem{Lemma}[Theorem]{Lemma}
\newtheorem{Example}[Theorem]{Example}

\newtheorem{Corollary}[Theorem]{Corollary}
\numberwithin{Theorem}{section}

\journal{Annals of Pure and Applied Logic} 

\begin{document}

\begin{frontmatter}

\title{Uniform Interpolation and Coherence\tnoteref{thanks}}

\author{Tomasz Kowalski}
\address{Department of Mathematics and Statistics, La Trobe University, Melbourne, Australia}
\ead{T.Kowalski@latrobe.edu.au}

\author{George Metcalfe}
\address{Mathematical Institute, University of Bern, Switzerland}
\ead{george.metcalfe@math.unibe.ch}

\tnotetext[thanks]{This project has received funding from the European Union's Horizon 2020 research and innovation programme under the Marie Sk{\l}odowska-Curie grant agreement No 689176.}

%%%%%%%%%%%%%%%%%%%%%%%%%%%%%%%%%%%%%%%%%%%%%%%

\begin{abstract}
A variety $\V$ is said to be coherent if any finitely generated subalgebra of a finitely presented member of $\V$ is  finitely presented. It is shown here that $\V$ is coherent if and only if it satisfies a restricted form of uniform deductive interpolation: that is, any compact congruence on a finitely generated free algebra of $\V$ restricted to a free algebra over a subset of the generators is again compact. A general criterion is obtained for establishing failures of coherence, and hence also of uniform deductive interpolation. This criterion is then used in conjunction with properties of canonical extensions to prove that coherence and uniform deductive interpolation fail for certain varieties of Boolean algebras with operators (in particular, algebras of modal logic $\lgc{K}$ and its standard non-transitive extensions), double-Heyting algebras, residuated lattices, and lattices. 
\end{abstract}


\end{frontmatter}

%%%%%%%%%%%%%%%%%%%%%%%%%%%%%%%%%%%%%%%%%%%%%%%%

\section{Introduction}

Uniform interpolation was established for intuitionistic propositional logic $\lgc{IPC}$ by Pitts in~\cite{Pit92} and  used by Ghilardi and Zawadowski in~\cite{GZ97} to prove that the first-order theory of Heyting algebras has a model completion. More generally, the  latter authors proved in~\cite{GZ02} that a model completion exists for the first-order theory of any variety $\V$ satisfying certain category-theoretic conditions. These conditions were reformulated by van Gool et al.~in~\cite{vGMT17} as properties of equational consequence in $\V$ --- most prominently, right and left uniform deductive interpolation --- and related to properties of compact congruences on free and finitely presented algebras of $\V$. In particular, if $\V$ admits deductive interpolation (or has the amalgamation property), then right uniform deductive interpolation for $\V$ amounts to the condition that the restriction of any compact congruence on a finitely generated free algebra of $\V$ to a free algebra over a subset of the generators is again compact. 

In Section~\ref{s:interpolation} of this paper, we prove that this last condition is equivalent to the model-theoretic notion of coherence considered by Wheeler in~\cites{Whe76,Whe78} and  also studied quite widely in algebra, mostly in connection with rings, groups, monoids, and lattices (see, e.g.,~\cites{ES70,CLL73,Sch83,Gou92}). A variety $\V$ is said to be coherent if every finitely generated subalgebra of a finitely presented member of  $\V$ is again finitely presented.\footnote{Note that the notion of coherence defined and  studied by Taylor in~\cite{Tay74} is entirely different, and not related to our results.} Coherence of $\V$ is implied by, and indeed in conjunction with amalgamation and another property implies, the existence of a model completion for the first-order theory of $\V$ (see~\cite{Whe76}).

Following Pitts' theorem for $\lgc{IPC}$, many proofs of uniform interpolation or its failure for various logics have appeared in the literature. In particular, all intermediate logics with Craig interpolation admit uniform interpolation; however, some modal logics, including $\lgc{S4}$ and $\lgc{K4}$, admit Craig interpolation but not uniform interpolation~\cites{GZ02,Bil07}. For the modal logic $\lgc{K}$, some extra care is necessary. A semantic proof of uniform interpolation was given for $\lgc{K}$ (also G{\"o}del-L{\"o}b logic $\lgc{GL}$ and Grzegorgczyk logic $\lgc{S4Grz}$) by Visser in~\cite{Vis96}, and a Pitts-style proof was provided (also for $\lgc{KT}$) by B{\'i}lkov{\'a} in~\cite{Bil07}. However, these proofs establish a uniform ``implication-based'' interpolation property, and not the uniform deductive ``consequence-based'' interpolation property considered in~\cite{vGMT17}.  The same observation applies to uniform interpolation results for substructural logics (varieties of residuated lattices) established by Alizadeh et~al.~in~\cite{ADO14}. Note in particular that the varieties of lattice-ordered abelian groups and MV-algebras admit uniform deductive interpolation, but not Craig interpolation (see~\cite{vGMT17}).

In Section~\ref{s:criterion} of this paper, we provide a general criterion for establishing the failure of coherence, and hence also of uniform deductive interpolation. This criterion states, roughly, that in a coherent variety $\V$ of algebras with a term-definable semilattice reduct, any decreasing and monotone term that satisfies a fixpoint embedding condition in $\V$ admits a fixpoint obtained by iterating the term finitely many times. In Section~\ref{s:canonical}, we review briefly the theory of canonical extensions and prove two useful fixpoint lemmas. In Section~\ref{s:casestudies}, we use these lemmas and the general criterion of Section~\ref{s:criterion} to obtain a condition for the failure of coherence for varieties of ordered algebras closed under canonical completions. We then use this condition to show that any coherent variety of Boolean algebras with operators that is closed under canonical extensions has equationally definable principal congruences (EDPC). In particular, $\lgc{K}$ is not coherent, does not admit uniform deductive interpolation, and its first-order theory does not have a model completion. Indeed, the same is true of any normal modal logic closed under canonical extensions for which $\boxdot^n x \eq \boxdot^{n+1} x$ fails for all $n \in \N$ (where $\boxdot x := \Box x\wedge x$). We obtain similar results also for varieties of residuated lattices, double-Heyting algebras, and lattices. In the latter case, we obtain an alternative proof of Schmidt's result that the variety of lattices is not coherent and its first-order theory does not have a model completion~\cite{Sch83}.

%%%%%%%%%%%%%%%%%%%%%%%%%%%%%%%%%%%%%%%%%%%%%%%%

\section{Uniform Deductive Interpolation and Coherence} \label{s:interpolation}

In this section we first recall the definitions of deductive interpolation and uniform deductive interpolation for equational consequence in a variety, and their algebraic characterizations in terms of congruences on free algebras. We then recall the notion of coherence for a variety and relate this notion to uniform deductive interpolation (Theorem~\ref{th:coherent}).

Let us assume that $\lang$ is an algebraic signature with at least one constant symbol and that $\V$ is a variety of $\lang$-algebras. The assumption that $\lang$ contains a constant is not essential --- indeed it will be dropped when considering varieties of lattices in Subsection 5.4 --- but is adopted here for convenience of presentation and easier reference to~\cites{MMT14,vGMT17}. 

For any (possibly infinite) set of variables $\x$, we denote by $\Tm(\x)$ the
\emph{$\lang$-term} {\em algebra over $\x$} and by $\F(\x)$, the {\em free
  algebra of $\V$ over $\x$}. We write $t(\x)$,  $\eps(\x)$, or $\Si(\x)$ to
denote that the variables of, respectively, an $\lang$-term $t$,
$\lang$-equation $\eps$, or set of $\lang$-equations $\Si$ are included in
$\x$. Where appropriate, we deliberately confuse these expressions with the
corresponding elements, pairs of elements, and sets of pairs of elements from
$\F(\x)$. We also adopt the convention that $\x$, $\y$, etc. denote disjoint
sets, and let $\x,\y$ denote their disjoint union.

For a set of $\lang$-equations $\Si \cup \{\eps\}$ containing exactly the variables in the set $\x$, define
\[
\begin{array}{rcc}
\Si \mdl{\V} \eps & :\Longleftrightarrow & \text{for every $\m{A} \in \V$ and homomorphism $e
\colon \Tm(\x) \to \m{A}$,}\\[.1in] 
& &  \Si \subseteq \ker(e) \ \Longrightarrow \ \eps \in \ker(e).
\end{array}
\]
For a set of $\lang$-equations $\Si \cup \De$, we write $\Si \mdl{\V} \De$ if $\Si \mdl{\V} \eps$ for all $\eps \in \De$. 

We say that $\V$ admits {\it deductive interpolation} if for any sets $\x,\y,\z$ and set of equations $\Si(\x,\y) \cup \{\eps(\y,\z)\}$ satisfying $\Si \mdl{\V} \eps$, there exists a set of equations $\Pi(\y)$ such that $\Si \mdl{\V} \Pi$ and $\Pi \mdl{\V} \eps$. This property has been studied in depth by many authors (see, e.g.,~\cites{Jon65,Pig72,Mak77,Wro84a,Cze85,Ono86,Cze07,MMT14}). In particular, it is known that if $\V$ has the amalgamation property, then it admits deductive interpolation and, conversely, if $\V$ admits deductive interpolation and has the congruence extension property, then it has the amalgamation property (see~\cite{MMT14} for proofs and further references). Let us just also note here (see~\cite{vGMT17} for a proof) that $\V$ admits deductive interpolation if and only if for any finite sets $\x,\y$ and finite set of equations $\Si(\x,\y)$, there exists a set of equations $\Pi(\y)$ such that for any equation $\eps(\y,\z)$, 
\[ 
\Si \mdl{\V} \eps \iff \Pi \mdl{\V} \eps.
\]
Following~\cite{vGMT17}, we say that $\V$ admits {\em right uniform deductive interpolation} if $\Pi(\y)$ in the preceding condition is required to be finite. 

To reformulate these notions via congruences on free algebras of $\V$, let us denote the congruence on an algebra $\m{A}$ generated by some $S \subseteq A^2$  by $\cg{\m{A}}(S)$, and recall (see~\cite{MMT14}*{Lemma~2}) that for any sets of equations $\Si(\x), \De(\x)$, 
\[
\Si \models_\V \De \iff \cg{\F(\x)} \De \subseteq \cg{\F(\x)} \Si.
\]
Let us also denote the congruence lattice of an algebra $\m{A}$ by $\Con \m{A}$, and recall that the \emph{adjoint lifting} of a homomorphism $h \colon \m A \to \m B$ in $\V$ to the congruence lattices of $\m{A}$ and $\m{B}$ is the adjoint pair of maps 
\begin{center}
$h^* \colon \Con \m{A} \leftrightarrows \Con \m{B} \cocolon h^{-1}$,
\begin{align*}
h^*(\psi) &:= \cg{\m B}\{ \lan h(a),h(a') \ran  \mid  \lan a,a' \ran
            \in \psi\}, \\ 
h^{-1}(\The) &:= \{\lan a,a' \ran \in A^2  \mid  \lan h(a),h(a')
                 \ran \in \The\} = \ker h(-)/{\The}. 
\end{align*}
\end{center}
It is easily checked that  $\V$  admits deductive interpolation if and only if for any finite sets $\x,\y,\z$, the following diagram commutes:
\begin{equation*}
\begin{tikzpicture}[baseline=(current  bounding  box.center)]
 \matrix (m) [matrix of math nodes, row sep=1.5em, column sep=1.5em]{
\Con \F(\x,\y) 		& 		& \Con  \F(\y) 		\\
 		&		& 		\\
\Con \F(\x,\y,\z) 	&     		& \Con \F(\y,\z) 	\\
};
\path
(m-1-1) edge[->] node[above] {$i^{-1}$} (m-1-3)
(m-1-1) edge[->] node[left] {$j^*$} (m-3-1)
(m-3-1) edge[->] node[below] {$k^{-1}$} (m-3-3)
(m-1-3) edge[->] node[right] {$l^*$} (m-3-3);
\end{tikzpicture}
\end{equation*}
where $i$, $j$, $k$, and $l$ are the inclusion maps between corresponding finitely generated free algebras. 

Let us denote the set of compact (finitely generated) congruences on an algebra $\m A$ by $\KCon \m{A}$, noting that $\KCon \m A$ is always a join-subsemilattice of $\Con \m A$, but meets in $\KCon \m A$ need not exist in general. For a homomorphism $h \colon \m A \to \m B$, the map $h^*$  restricts to a map $h^*|_{\KCon \m{A}} \colon \KCon \m A \to \KCon \m B$, which we call the \emph{compact lifting} of $h$. On the other hand, $h^{-1}$ restricts to $h^{-1}|_{\KCon \m{B}} \colon \KCon \m B \to \KCon \m A$, the right adjoint of $h^*|_{\KCon \m{A}}$, if and only if $h$ preserves compact congruences. The next result shows that if such adjoints exist for all inclusion maps between finitely generated free algebras in $\V$, then they exist for all homomorphisms between finitely presented algebras in $\V$.

\begin{Proposition}[cf.~\cite{vGMT17}*{Proposition~3.8}]\label{p:characterization}
The following are equivalent:
\begin{newlist}
\item[\rm (1)] 	For any finite sets $\x,\y$ and finite set of equations
  $\Si(\x,\y)$, there exists a finite set of equations $\Pi(\y)$ such that for any
  equation $\eps(\y)$, 
\[ 
\Si \mdl{\V} \eps \iff \Pi \mdl{\V} \eps.
\]
\item[\rm (2)] 	For any finite sets $\x,\y$ and compact congruence $\The$ on $\F(\x,\y)$, the congruence $\The \cap F(\y)^2$  on $\F(\y)$ is compact.
\item[\rm (3)]	For any finite sets $\x,\y$, the compact lifting of the inclusion map from $\F(\y)$ to $\F(\x,\y)$ has a right adjoint.
\item[\rm (4)] 	The compact lifting of any homomorphism between finitely
  presented algebras in $\V$ has a right adjoint. 
\end{newlist}
\end{Proposition}

Following Wheeler~\cite{Whe76}, let us call $\V$ \emph{coherent} if every finitely generated subalgebra of a finitely presented member of $\V$ is itself finitely presented. It is proved in~\cite{Whe76} that the coherence of $\V$ is implied by (and in conjunction with amalgamation and another property, implies) the existence of a model completion for the first-order theory of $\V$. Note that, by our earlier assumption, coherence is defined here only for varieties in a signature $\lang$ that contains at least one constant symbol. This restriction is not essential, and in fact we will remove it when considering lattices in Subsection~\ref{ss:lattices}, but allows for a neater presentation.

Below we establish a useful technical result, proved in a slightly different form as Lemma~3.9 in~\cite{vGMT17}.

\begin{Lemma}\label{l:con-gen}
Suppose that  $f\colon \m{F}(\y)\to \m{A}$ and  $g\colon \m{F}(\x)\to \m{A}$ are surjective homomorphisms 
in $\V$ and let $r \colon \m{F}(\y)\to \m{F}(\x)$ and $s\colon \m{F}(\x)\to \m{F}(\y)$  be the natural maps satisfying 
$f = g\circ r$ and $g = f\circ s$. If $\ker(g)$ is generated by  $\Pi \subseteq F(\x)^2$, then $\ker(f)$ is generated by $\Si = \{\lan s(a), s(b)\ran \mid \lan a,b\ran\in \Pi\}\cup\{\lan y,sr(y)\ran \mid y\in\y\} \subseteq F(\y^2)$.
\end{Lemma}  

\begin{proof}
The situation is depicted in the following diagram:
\begin{figure}[h]
\begin{center}  
\begin{tikzpicture}[->, >=stealth,auto]
\node (A) at (0,0) {$\m{A}$}; 
\node (Fx) at (2,1) {$\F(\overline{x})$};
\node (Fy) at (2,-1) {$\F(\overline{y})$};
\node (B) at (0,-1.5) {$\m{B}$};
\path (Fx) edge[->>] node[above] {$g$} (A) (Fy) edge[->>] node[above] {$f$} (A)
(Fy) edge[bend left] node {$r$} (Fx) (Fx) edge[bend left] node {$s$} (Fy);
\path (Fy) edge[->>] node {$p$} (B) (A) edge[dashed] node {$k$} (B);
\end{tikzpicture}
\end{center}
\end{figure}

\noindent
Observe first that $\Si\subseteq \ker(f)$. For any $\lan a,b\ran\in \Pi$, we have $g(a) = g(b)$ and so $fs(a) = g(a) = g(b) = fs(b)$; that is, $\lan s(a),s(b) \ran \in \ker(f)$. Also, given any $y\in\y$, we have $f(y) =gr(y) = fsr(y)$; that is, $\lan y,sr(y) \ran \in \ker(f)$. 

Now let $\Psi=\cg{\m{F}(\y)}{\Si}$. Define $\m{B} = \m{F}(\y)/\Psi$ and let $p\colon \m{F}(\y)\to\m{B}$ be the natural homomorphism  with $\Psi = \ker(p) \subseteq \ker(f)$. Let $h = p\circ s$, and observe that for $\lan a,b\ran \in \Pi$, we have $h(a) =  ps(a) = ps(b) = h(b)$, since $\lan s(a),s(b)\ran\in \Psi$. This implies $\ker(g) \subseteq\ker(h)$, and hence there exists a unique homomorphism $k\colon \m{A}\to\m{B}$ such that $k\circ g = h$.  For each $y\in\y$,  
\[
p(y) = psr(y) = hr(y) = kgr(y) = kf(y),
\]
which, by freeness, implies that $p = kf$.   But then $\Psi = \ker(p) \supseteq \ker(f)$, establishing the desired equality $\Psi=\ker(f)$. 
\end{proof}  

Note that if $\Pi$ and $\y$ in Lemma~\ref{l:con-gen} are finite, then $\Si$ is also finite. Other properties of $\Pi$, such as being recursive, will also transfer to $\Si$ under certain further mild assumptions, but this will not concern us here. 

We now prove the main result of this section.

\begin{Theorem}\label{th:coherent}
$\V$ is coherent if and only if any of the equivalent conditions of Proposition~\ref{p:characterization} holds. 
\end{Theorem}

\begin{proof}
Assume that $\V$ is coherent. We will prove that condition (2) of Proposition~\ref{p:characterization} holds.  Let $\The$ be a compact congruence on $\m{F}(\y,\z)$, so that $\m{C} = \m{F}(\y,\z)/\The$ is finitely presented. Let $\Psi = \The\cap F(\y)^2$ and  $\m{A} = \m{F}(\y)/\Psi$. Then $\m{A}$ is finitely generated and embeds into $\m{C}$. By coherence, $\m{A}$ is finitely presented.  Hence let $\m{F}(\x)/\Phi \cong \m{A}$ be a finite presentation such that $\Pi$ is a finite set of generators of $\Phi$. Let  $f\colon \m{F}(\y)\to \m{A}$ and  $g\colon \m{F}(\x)\to \m{A}$ be the surjective homomorphisms such that $\Psi = \ker(f)$ and $\Phi=\ker(g)$, and let $r \colon \m{F}(\y)\to \m{F}(\x)$ and $s\colon \m{F}(\x)\to \m{F}(\y)$ be the natural maps satisfying $f = g\circ r$ and $g = f\circ s$. Then the assumptions of Lemma~\ref{l:con-gen} are satisfied, so $\Psi$ is generated by a finite set. That is, $\Psi$ is compact, as required.

For the converse,  assume that $\V$ is not coherent. Then there exists a finitely presented algebra $\m{B}$ in $\V$ and a subalgebra $\m{A}$ of $\m{B}$ that is finitely generated but not finitely presented.  Let $(G_B,R_B)$ be a finite presentation of $\m{B}$, and let $G_A$ be a finite set of generators of $\m{A}$. We construct another finite presentation of $\m{B}$ as follows. We let $G = G_B\cup G_A$ be the set of generators. Since each $g\in G_A$ is generated from $G_B$, we have $g = t_g(G_B)$ for some term $t_g$. Let $R'$ be the set of witnessing relations $\lan g,t_g(G_B) \ran$ for each $g\in G_A$. We let $R = R_B\cup R'$ and obtain a presentation $(G,R)$ of $\m{B}$ that is still finite, since $G_A$ is finite. Now consider the free algebra $\m{F}(\x,\y)$ such that there exist bijections between $\x$ and $G_B$, and $\y$ and $G_A$. The kernel of the induced homomorphism from $\m{F}(\x,\y)$ onto $\m{B}$ is a compact congruence $\The$ on $\m{F}(\x,\y)$. However, $\Psi = \The\cap F(\y)^2$ is not compact, as otherwise $\m{F}(\y)/\Psi$ would give a finite presentation of $\m{A}$, contradicting the assumption.  Hence condition (2) of Proposition~\ref{p:characterization} fails.
\end{proof}

Recall that Higman's embedding theorem for groups (cf.~\cite{Hig61}) states that every  finitely generated recursively presented group embeds into some finitely presented group. Following~\cite{KS95}, we say that $\V$ has the \emph{Higman property} if every finitely generated recursively presented algebra in $\V$ embeds into a finitely presented algebra in $\V$. 

\begin{Proposition}\label{p:Higman}
If every finitely generated recursively presented algebra in $\V$ is finitely presented, then $\V$ is coherent. Moreover, if $\V$ satisfies the Higman property, then the converse also holds. 
\end{Proposition}  

\begin{proof}
First we prove that a certain converse to the Higman property holds: namely, if $\m{A}$ is a finitely generated subalgebra of some finitely presented  $\m{B}\in\V$, then $\m{A}$ is recursively presented.  As in the proof of Theorem~\ref{th:coherent},  we may assume without loss of generality that the set of generators of $\m{A}$ is contained in the set of generators of $\m{B}$. Suppose then that $\m{B}\cong\m{F}(\x,\y)/\The$ for some compact congruence $\The$ and $\m{A}\cong \m{F}(\x)/\Psi$, where $\Psi = \The\cap\m{F}(\x)^2$. Since $\The$ is compact and $\x$ is finite, $\Psi$ is recursively generated. Hence $\m{A}$ is recursively presented as claimed.

Now assume that every finitely generated recursively presented algebra in $\V$ is finitely presented, and consider a finitely generated subalgebra $\m{A}$ of some finitely presented $\m{B}\in\V$. As we have just shown, $\m{A}$ must be recursively presented, and hence $\m{A}$ is finitely presented. The remaining part is clear. 
\end{proof}

\begin{Example}\label{e:typical}
Clearly, every locally finite variety is coherent. Less obviously, the property holds for the varieties of Heyting algebras (the main content of Pitts' theorem for $\lgc{IPC}$~\cite{Pit92}), abelian groups, lattice-ordered abelian groups, and MV-algebras (see~\cite{vGMT17}). On the other hand, by Higman's embedding theorem, the variety of groups is not coherent, since there exists a finitely generated recursively presented group that is not finitely presented. Similar reasoning for monoids and other varieties possessing the Higman property (see~\cite{KS95}) produces further failures of coherence.
\end{Example}

Combining Theorem~\ref{th:coherent} with \cite{vGMT17}*{Proposition~3.5}, we obtain also the following characterization of right uniform deductive interpolation.

\begin{Proposition}
$\V$ admits right uniform deductive interpolation if and only if $\V$ is coherent and admits deductive interpolation.
\end{Proposition}

A similar characterization has been obtained for a left uniform deductive interpolation property (see~\cite{vGMT17}*{Proposition~4.3}). However, in this paper, we focus only on failures of right uniform deductive interpolation, indeed only on cases where coherence fails. 

%%%%%%%%%%%%%%%%%%%%%%%%%%%%%%%%%%%%%%%%%%%%%%%%

\section{A General Criterion} \label{s:criterion}

The main result of this section, Theorem~\ref{t:general}, establishes that in a coherent variety, unary terms satisfying certain conditions also satisfy an $n$-potency identity for some $n \in \N$. Understood contrapositively, this result provides a general criterion for demonstrating the failure of coherence in a variety. 

Let $\lang$ be a signature containing at least one constant symbol and let $\V$ be a variety of $\lang$-algebras. For any unary $\lang$-term $t$, we define inductively
\[
t^0(x) = x \quad \text{and} \quad t^{k+1}(x) = t(t^{k}(x)) \ \text{ for } k \in \N.
\]
We say that $t$ is {\em $n$-potent} (for $n \in \N$) in $\V$ if $\V \models t^{n+1}(x) \eq t^{n}(x)$. 

\begin{Theorem}\label{t:general}
Let $\V$ be a coherent variety of $\lang$-algebras with a meet-semilattice  term-definable reduct and a term $t(x)$ satisfying
\[
\V \models t(x) \le x \quad \text{ and } \quad \V \models x \le y\, \Rightarrow\,  t(x) \le t(y). 
\]  
Suppose also that $\V$ satisfies the following \emph{fixpoint embedding condition} with respect to $t(x)$:

\begin{enumerate}
\item[\textup{(FE)}]
For any finitely generated $\m{A}\in \V$ and $a \in A$, there exists an algebra $\m{B}\in\V$ such that $\m{A}$ is a subalgebra of $\m{B}$ and the join $\bigwedge_{k \in \N} t^k(a)$ exists in $\m{B}$ and satisfies  
\[
 \bigwedge_{k \in \N} t^k(a) = t(\bigwedge_{k \in \N} t^k(a)).
\]
\end{enumerate}
Then $t$ is $n$-potent in $\V$ for some $n \in \N$.
\end{Theorem}
\begin{proof}
Let $\V$ and $t(x)$ be as in the statement of the theorem, and  define
\[
\Si = \{y \le x, x \le z, x \eq t(x)\}
\quad
\mbox{and}
\quad
\Pi = \{y \le t^k(z)  \mid k \in \N\}.
\]
We prove that for any equation $\eps(y,z)$,
\[
\Si \mdl{\V} \eps(y,z) \iff \Pi \mdl{\V} \eps(y,z).
\]
For the right-to-left direction, it suffices to observe that $\Si \mdl{\V} y \le t^k(z)$ for each $k \in \N$, and hence $\Si \mdl{\V} \Pi$. For the converse direction, suppose contrapositively that $\Pi \not \mdl{\V} \eps(y,z)$. Since only two variables occur in $\Pi$, there exist a finitely generated $\m{A} \in \V$ and a homomorphism $e \colon \Tm(y,z) \to \m{A}$ such that $\Pi \subseteq \ker(e)$, but $\eps\not \in \ker(e)$. Let $a=e(y)$.  By assumption, $\m{A}$ is a subalgebra of some $\m{B}\in \V$ such that $\bigwedge_{k \in \N} t^k(a)$ exists in $\m{B}$ and satisfies  
\[
 \bigwedge_{k \in \N} t^k(a) = t(\bigwedge_{k \in \N} t^k(a)).
\]
Since $x$ does not appear in $\Pi \cup\{\eps\}$, we may extend $e$ to a homomorphism $e \colon \Tm(x,y,z) \to \m{B}$ by
\[
e(x) = \bigwedge_{k \in \N} t^k(b).
\]
We have $e(y) \le t^k(b)$ for each $k \in \N$, so clearly $e(y) \le e(x) \le e(z)$. Moreover,  by assumption, 
\[
e(x) = \bigwedge_{k \in \N} t^k(b) =  t(\bigwedge_{k \in \N} t^k(a)) = e(t(x)).
\]
Hence $\Si \subseteq \ker(e)$ and we obtain $\Si \not \mdl{\V} \eps(y,z)$. 

Finally, since $\V$ is coherent, by Theorem~\ref{th:coherent}, there exists a finite set of equations $\De(y,z)$ such that for any equation $\eps(y,z)$,
\[
\Si  \mdl{\V} \eps(y,z) \iff \De \mdl{\V} \eps(y,z).
\]
In particular, $\Si \mdl{\V} \De$, and so, by the above implication, $\Pi \mdl{\V} \De$. Moreover, using compactness and the fact that $\De$ is finite, $\Pi' \mdl{\V} \De$ for some finite $\Pi' \subseteq \Pi$. But also, since $t$ is decreasing, $\{y \le t^{k+1}(z)\} \mdl{\V} y \le t^k(z)$ for each $k \in \N$. Hence for some particular $n \in \N$, we have $\{y \le t^n(z)\} \mdl{\V} \De$. Recall that $\Si  \mdl{\V} y \le t^{n+1}(z)$, and so also $\De  \mdl{\V} y \le t^{n+1}(z)$. Combining consequences, we obtain $\{y \le t^n (z)\}\mdl{\V} y \le t^{n+1}(z)$. Finally, substituting $y$ with $t^n(z)$, we obtain $\mdl{\V} t^n(z) \le t^{n+1}(z)$. That is, $t$ is $n$-potent in $\V$.
\end{proof}

Let us mention that the proof of Theorem~\ref{t:general} can be used to obtain direct counterexamples to coherence. Suppose that $\V$ satisfies the conditions of the theorem and that $t$ is not $n$-potent in $\V$ for some $n \in \N$. Let $\The$ be the compact congruence on $\F(x,y,z)$ generated by $\{y \le x, x \le z, x \eq t(x)\}$, and let $\Psi$ be the congruence on $\F(y,z)$ generated by $\{y \le t^n(z)  \mid n \in \N\}$. Then $\F(y,z) / \Psi$ is a finitely generated but not finitely presented member of $\V$ that embeds into the finitely presented algebra $\F(x,y,z) / \The$ in $\V$. 

%%%%%%%%%%%%%%%%%%%%%%%%%%%%%%%%%%%%%%%%%%%%%%%%

\section{Canonical extensions}\label{s:canonical}

In this section, we describe a second tool for establishing the failure of coherence and uniform deductive interpolation for a variety: the theory of canonical extensions. To keep the paper reasonably self-contained, we begin with a brief review of this theory based on the development in~\cite{GH01}. The reader familiar with canonical extensions may skip this section, with the exception of Lemmas~\ref{fixpoint} and~\ref{exp-fix}, which are required for the applications in Section~\ref{s:casestudies}, and do not appear in~\cite{GH01} or, as far as we can tell, elsewhere in the literature.  

A \emph{completion} of a poset $\m{P}$ is a pair $\lan e,\m{C} \ran$, where $\m{C}$ is a complete lattice and $e$ is an order embedding of $\m{P}$ into $\m{C}$ that preserves all existing finite meets and joins of $\m P$. An element $a\in C$ is called \emph{open} if $a=\bigvee e(X)$ for some  subset $X$ of $P$, where the join is taken in $C$; note that in this case $X$ can be taken to be the set $\{x \in P \mid a \leq e(x)\}$. Dually, $a\in C$ is \emph{closed} if $a=\bigwedge e(X)$ for some $X \subseteq P$. We will use $K$ and $O$ to denote the sets of closed and open elements of $\m{C}$, respectively.  A completion $\lan e, \m{C} \ran$ is called
\begin{newlist}
\item[$\bullet$] \emph{dense} if every element of $C$ is both a join of closed elements and a meet of open elements;
\item[$\bullet$]  \emph{compact} if for any $A\subseteq K$ and $B\subseteq O$, we have $\bigwedge A\leq\bigvee B$ if and only if there are finite subsets $A_0$ of $A$ and $B_0$ of $B$ satisfying $\bigwedge A_0\leq\bigvee  B_0$. 
\end{newlist}
A dense and compact completion $\lan e, \m{C} \ran$ of $\m{P}$  is called a \emph{canonical extension}. 

\begin{Theorem}
Any poset\/ $\m{P}$ has a canonical extension $\lan e,\m{C} \ran$. Moreover, if $\lan e',\m{C}' \ran$ is another canonical extension of\/ $\m{P}$, then there exists a lattice isomorphism $i\colon \m{C}'\to \m{C}$ such that $i \circ e' = e$.
\end{Theorem}

Following standard practice, from now on we will speak of \emph{the} canonical extension of $\m{P}$, denoted by $\m{P}^\si$. We will also assume that the embedding is realised as the natural identity embedding, so that $\m{P}\in \op{S}(\m{P}^\si)$.

Maps between posets also have canonical extensions. Let $\m{P}$ and $\m{Q}$ be posets, and let $f\colon P\to Q$ be any map. The maps $f^\si, f^\pi\colon P^\si\to Q^\si$ are defined as follows:
\begin{align*}
f^\si(x) &= \bigvee\bigl\{\bigwedge\{f(a)\mid a\in P, \ p\leq a\leq q\} \mid  p\in K, q\in O, \ p\leq x\leq q\bigr\};\\
f^\pi(x) &= \bigwedge\bigl\{\bigvee\{f(a)\mid a\in P, \ p\leq a\leq q\} \mid p\in K, q\in O, \ p\leq x\leq q\bigr\}.
\end{align*}
The following two lemmas are easy consequences of these definitions. 

\begin{Lemma}\label{sigma-pi-ext}
Both $f^\si$ and $f^\pi$ extend $f$. Moreover, $f^\si\leq f^\pi$ under the pointwise ordering.
\end{Lemma}

\begin{Lemma}\label{ord-pres-maps}
Let $f\colon P\to Q$ be an order-preserving map.  
\begin{newlist} 
\item[\rm (a)] $f^\si(p) = \bigwedge\{f(a) \mid a\in P,\ p\leq a\}$, for all $p\in K$;
\item[\rm (b)] $f^\pi(q) = \bigvee\{f(a) \mid a\in P,\ q\geq a\}$, for all $q\in O$;
\item[\rm (c)] $f^\si(x) = \bigvee\{f^\si(p) \mid p\in K,\ p\leq x\}$, for all $x\in P^\si$;
\item[\rm (d)] $f^\pi(x) = \bigwedge\{f^\pi(q) \mid q\in O,\ q\geq x\}$, for all $x\in P^\si$;
\item[\rm (e)] $f^\si$ and $f^\pi$ are equal on $K\cup O$.
\end{newlist}
\end{Lemma}
If $f^\si = f^\pi$, then we say that $f$ is \emph{smooth}. An example of a non-smooth map is the implication on a Heyting algebra, viewed as a binary map from $A^\partial\times A$ to $A$. The same holds for the residuals of any order-preserving multiplication, so definitions of canonical extensions of residuated structures (see Section~\ref{s:casestudies}) must take this into account. To be more precise, 
\[
y\leq^\si x\ld^\pi z \iff
x\cdot^\si y\leq^\si z \iff 
x\leq^\si z\rd^\pi y,
\]
but these equivalences fail for $\ld^\si$ and $\rd^\si$. This example also illustrates how to obtain canonical extensions of posets with additional algebraic structure: since an $n$-ary operation $f$ on a poset $\m{P}$ is a map $f \colon P^n\to P$, we naturally obtain extensions $f^\si \colon (P^n)^\si\to P^\si$ and $f^\pi \colon (P^n)^\si\to P^\si$. Extensions $\si$ and $\pi$ commute with homomorphic images, substructures and finite direct products, so in particular, $(\m{P}^n)^\si = (\m{P}^\si)^n$ and $(\m{P}^n)^\pi = (\m{P}^\pi)^n$ for any $n \in \N$ and hence
canonical extensions of operations are computed coordinatewise. 

\begin{Lemma}\label{HSPfin-commut}
Let $\K$ be a class of algebras with term-definable poset reducts. Then canonical extensions (both $\si$ and $\pi$) of algebras from $\K$ commute with homomorphic images, subalgebras, and finite direct products.  
\end{Lemma}

If an operation is order-preserving in some coordinates and order-in\-ver\-ting in others, then it is often necessary to mix and match $\m{P}$ with $\m{P}^\partial$ accordingly, as shown by the residuation example above. For our purposes in this article, the maps obtained in this way are all we need. We will call them \emph{isotone} from now on, and treat them simply as order-preserving in each coordinate, trusting the reader to work out the appropriate dualisations.  

Extensions of arbitrary maps do not behave well under composition, but extensions of isotone maps are better behaved.

\begin{Lemma}\label{isotone}
Let $\m{P}$ be a poset, and let $f\colon P^n\to P$ and 
$g_1,\dots, g_n\colon P^k\to P$ be isotone maps. Then
\begin{newlist}
\item[\rm (a)] $(f(g_1,\dots, g_n))^\si\leq f^\si(g_1^\si,\dots, g_n^\si) \leq f^\si(g_1^\pi,\dots, g_n^\pi)$;
\item[\rm (b)] $f^\pi(g_1^\si,\dots, g_n^\si)\leq f^\pi(g_1^\pi,\dots, g_n^\pi)\leq (f(g_1,\dots, g_n))^\pi$.  
\end{newlist}
\end{Lemma}

We will call a class of algebras $\C$ a class of \emph{semilattice-ordered algebras} if the following conditions hold:
\begin{newlist}
\item[\rm (i)] Every algebra in $\C$ has a (uniformly) term-definable semilattice reduct.
\item[\rm (ii)] Each operation $o$ in the signature of $\C$ has a 
canonical extension, $o^\si$ or $o^\pi$, determined by $\C$.
\end{newlist}  
If $\C = \{\m{A}\}$, then we just call $\m{A}$ a {\em semilattice-ordered algebra}. For any $\m{A}\in\C$,  we let $\m{A}^\si$ denote the universe of the canonical extension of the poset reduct of $\m{A}$ equipped with the canonical extensions of the operations of $\m{A}$.  

Let $\C$ be a class of semilattice-ordered algebras. Following~\cite{Jon94}, a term $t$ will be called 
\begin{newlist}
\item[$\bullet$] $\C$-\emph{expanding} if $t^{\m{A}^\si} \geq (t^{\m{A}})^\si$, for all $\m{A}\in\C$;
\item[$\bullet$] $\C$-\emph{contracting} if $t^{\m{A}^\si} \leq (t^{\m{A}})^\si$, for all $\m{A}\in\C$;
\item[$\bullet$] $\C$-\emph{stable} if $t^{\m{A}^\si} = (t^{\m{A}})^\si$, for all $\m{A}\in\C$.
\end{newlist}
For a single algebra $\m{A}$, we will write $\m{A}$-expanding (contracting, stable), instead of the formally correct $\{\m{A}\}$-expanding, etc. By analogy, we will use $\C$-isotone and $\m{A}$-isotone, to mean isotone on each member of $\C$, and isotone on each member of $\{\m{A}\}$, i.e., isotone in $\m{A}$. Note that, despite the analogy, isotonicity is a property a term has with respect to a single class, whereas being expanding, contracting, or stable are properties a term has with respect to a pair of classes: the class $\C$ and the class $\C^\si = \{\m{C}^\si \mid \m{C}\in\C\}$. Often, $\C$ will be clear from the context and omitted. These definitions can obviously be extended to cover $\pi$ extensions and mixed cases, but  are not needed here.

Since a term $t$ is isotone if and only if it is a composition of isotone basic
operations, the first inequality in Lemma~\ref{isotone}(a) implies that all
isotone terms are expanding. Not all isotone terms are stable, but several
important ones are: e.g., the lattice operations and the Boolean complement.
The next lemma makes this observation precise.   

\begin{Lemma}\label{meet-is-meet}
For a lattice $\m{L}$, the extensions $\mt^\si$ and $\mt^\pi$ are equal to the
meet in $\m{L}^\si$. Similarly, $\jn^\si$ and $\jn^\pi$ are equal to the join in
$\m{L}^\si$. If $\m{L}$ is distributive, then so is $\m{L}^\si$. If $\m{B}$ is a
Boolean algebra, then so is $\m{B}^\si$; moreover, $\neg^\si$ and $\neg^\pi$ are
both equal to the Boolean complement in $\m{B}^\si$. 
\end{Lemma}

\begin{Corollary}\label{clone}
For any class $\C$ of semilattice-ordered algebras, the set of  $\C$-expanding terms is a clone.
\end{Corollary}  

Let us also recall some basic facts about operators. A map $f\colon  P^n \to P$ is an \emph{operator} if it preserves existing finite joins in each coordinate. A map $g$ is a \emph{dual operator} if it preserves existing finite meets in each coordinate. Recall that these definitions implicitly incorporate appropriate dualisations of coordinates. In particular,  the implication of a Heyting algebra $\m{A}$ is a dual operator when considered as a map from $A^\partial\times A$ to $A$. Operators are also called \emph{additive operators}, and dual operators, \emph{multiplicative operators}, in particular, in the context of Boolean algebras with operators.

\begin{Lemma}\label{operators}
Let $\m{P}$ be a poset and let $f\colon P^n\to P$ be an operator and $g_1,\dots g_n\colon P^k\to P$ isotone maps. Assume that the dualisations of the coordinates that make $f$ an operator agree with those that make $g_1,\dots g_n$ order-preserving. Then
\begin{newlist}
\item[\rm (a)] $f^\si$ preserves arbitrary non-empty joins in each coordinate;
\item[\rm (b)] $f^\si$ preserves upward directed joins;
\item[\rm (c)] $(f(g_1,\dots, g_n))^\si = f^\si(g_1^\si,\dots g_n^\si)$. 
\end{newlist}
The dual statements hold for dual operators. 
\end{Lemma}

We end this section with two fixpoint lemmas that will be crucial for the applications in Section~\ref{s:casestudies}. 

\begin{Lemma}\label{fixpoint}
Let $\m{P}$ be a poset and let $f\colon P\to P$ be an order-preserving map. If $X\subseteq P$ is downward directed and closed under $f$, and $f$ is decreasing on $X$, then  $f^\si(\bigwedge X) = \bigwedge X$ in $\m{P}^\si$.
\end{Lemma}

\begin{proof}
Let $y = \bigwedge X$ in $\m{P}^\si$. Since $X$ is closed under $f$, we have $f(x) \geq y$ for each $x\in X$.  As $y\in K$, by Lemma~\ref{ord-pres-maps}(a), we get $f^\si(y) = \bigwedge\{f(a) \mid a\in P,\ y\leq a\}$.  By compactness, for each $a\in P$ with $a\geq y$, there exists a finite $X_a\subseteq X$ with $a\geq \bigwedge X_a$, and since $X$ is downward directed, there is an element $x_a\in X$ such that $a\geq x_a$. So for every $a\in P$ with $a\geq y$, we have $a\geq x_a\geq y$ for some $x_a\in X$. Hence $\bigwedge\{f(a) \mid a\in P,\ y\leq a\} = \bigwedge\{f(x) \mid x\in X\}$. Now $X$ is closed under $f$, so $\bigwedge\{f(x) \mid x\in X\}\geq \bigwedge X$, but on the other hand, $f$ is decreasing on $X$, so $\bigwedge X \geq \bigwedge\{f(x) \mid x\in X\}$. Hence $f^\si(\bigwedge X) = \bigwedge X$, as claimed.
\end{proof}  

\begin{Lemma}\label{exp-fix}
Let $\m{A}$ be a semilattice-ordered algebra, and $t$ a unary term that is $\m{A}$-isotone and $\m{A}$-expanding. If $X\subseteq A$ is downward directed and closed under $t^\m{A}$, and $t^\m{A}$ is decreasing on $X$, then $t^{\m{A}^\si}(\bigwedge X) = \bigwedge X$  in $\m{A}^\si$. 
\end{Lemma}

\begin{proof}
By Lemma~\ref{fixpoint}, we have $(t^\m{A})^\si(\bigwedge X) = \bigwedge X$. Since $t$ is expanding, $t^{\m{A}^\si}(\bigwedge X) \geq \bigwedge X$. But $t^{\m{A}^\si}(x) = t^{\m{A}}(x)$ for every $x\in X$, as $X\subseteq A$. So $t^{\m{A}^\si}(x)\leq x$ holds for all $x\in X$, and $t^{\m{A}^\si}(\bigwedge X) \leq \bigwedge X$, since $t$ is isotone. This establishes the desired equality.
\end{proof}

%%%%%%%%%%%%%%%%%%%%%%%%%%%%%%%%%%%%%%%%%%%%%%%%

\section{Case Studies}\label{s:casestudies}

Theorem~\ref{t:general} can be roughly restated in the following form: if a coherent variety $\V$ of semilattice-ordered algebras is closed under (some) completions,  then every monotonic unary operation term-definable in $\V$ is either $n$-potent for some $n \in \N$ or non-continuous (fails to preserve meets of powers). The varieties of semilattice-ordered algebras mentioned in Example~\ref{e:typical} all fit this pattern. All  unary operations term-definable in a locally finite variety are $n$-potent for some $n \in \N$, and the same holds for all term-definable order-preserving decreasing (or increasing) unary operations of Heyting algebras. The varieties of lattice-ordered abelian groups and MV-algebras are coherent and have monotonic operations that are not $n$-potent for any $n \in \N$, but are not closed under completions. 

In this section, we use Theorem~\ref{t:general} to establish the failure of coherence and hence uniform deductive interpolation for various varieties of semilattice-ordered algebras that are closed under canonical extensions. These case studies are all corollaries of the following result.

\begin{Theorem}\label{t:basic}
Let $\V$ be a coherent variety of semilattice-ordered algebras that is closed
under canonical extensions. Then any unary $\V$-expanding term $t$ that is
order-preserving and decreasing in $\V$ is  $n$-potent for some $n \in \N$.   
\end{Theorem}   

\begin{proof}
Let $t$ be any unary $\V$-expanding term $t$ that is
order-preserving and decreasing in $\V$.
By Theorem~\ref{t:general}, it suffices to show that $\V$ satisfies the
fixpoint embedding condition (FE) with respect to $t$. Let $\m{A}$ be an algebra
in $\V$ and let $a\in A$. Then $\m{A}$ embeds into its canonical extension
$\m{A}^{\!\si}$, and by closure under canonical extensions we have
$\m{A}^{\!\si}\in\V$. We identify $\m{A}$ with its isomorphic copy in
$\m{A}^{\!\si}$, and let $X = \{t^k(a) \mid k\in\N\}$. Since $\m{A}^{\!\si}$ is complete,
$\bigwedge X$ exists  in $\m{A}^{\!\si}$.\footnote{Note that $\bigwedge X$ may
also exist in $\m{A}$, and then
$\bigwedge^{\m{A}^{\!\si}} X\geq \bigwedge^{\m{A}} X$. In general,
$t(\bigwedge^{\m{A}^{\!\si}} X) = \bigwedge^{\m{A}^{\!\si}} X$ does not imply
$t(\bigwedge^{\m{A}} X) = \bigwedge^{\m{A}} X$, so the required fixpoint
for $t$ may always ``escape'' to the canonical extension.}     
By assumption, $t$ is 
$\m{A}$-expanding and $\m{A}$-isotone, so Lemma~\ref{exp-fix} applies, 
yielding $t^{\m{A}^\si}(\bigwedge X) = \bigwedge X$. This shows that
(FE) holds, as required.
\end{proof}

%%%%%%%%%%%%%%%%%%%%%%%%%%%%%%%%%%%%%%%%%%%%%%%%

\subsection{Varieties of Boolean algebras with operators}

A {\em Boolean algebra with operators} is an algebra $\m{A} = \lan A, \mt,\jn,\neg,0,1, \mathcal{O}\ran$ such that $\lan A, \mt,\jn,\neg,0,1\ran$ is a Boolean algebra, and $\mathcal{O}$ is a set of (multiplicative) operators. In this section, we will refer to $\mathcal{O}$ as the \emph{signature} of $\m{A}$, assuming tacitly that the Boolean operations are always present.  We will also refer to operations in $\mathcal{O}$ simply as operators. (This is terminologically at odds with previous sections, but follows standard practice.)

Let $\mathrm{Clo_{ex}}(\mathcal{O})$ be the clone of all expanding terms over the signature $\mathcal{O}$.
For $f\in\mathcal{O}$, define $\boxdot^f_k x = f(\overline{0},x,\overline{0})\mt x$, with $x$ occurring only at coordinate $k$, and $\boxdot^f x = \bigwedge_{0<i\leq k} \boxdot^f_i x$. If $\mathcal{O}$ is finite, then we also let $\boxdot x = \bigwedge_{f\in \mathcal{O}} \boxdot^f x$. All these terms interpret to operators, so they belong to $\mathrm{Clo_{ex}}(\mathcal{O})$ by Lemma~\ref{operators}. 

It is known that a variety $\V$ of Boolean algebras with operators of finite signature $\mathcal{O}$ admits \emph{equationally definable principal congruences} (EDPC) if and only if $\boxdot$ is $n$-potent for some $n \in \N$  (see, e.g.,~\cite{KK06}). We will now show that for any variety $\V$ of Boolean algebras with operators, coherence implies EDPC for all varieties inheriting finitely many operators from $\V$, even if $\V$ itself is of infinite signature.

\begin{Theorem}\label{t:BAOs} 
Let $\V$ be a variety of Boolean algebras with operators of signature $\mathcal{O}$ that is closed under canonical extensions. Let $\mathcal{O}'$ be a finite subset of $\mathrm{Clo_{ex}}(\mathcal{O})$, and let $\V'$ be the variety generated by term-reducts of members of $\V$ in the signature $\mathcal{O}'$. If $\V$ is coherent, then  $\V'$ has the EDPC. 
\end{Theorem}

\begin{proof}
Let $\C = \{\m{A}^{\!\si} \mid \m{A}\in\V\}$, so that $\V =
\op{IS}(\C)$. Defining $\C'$ to be the class of $\mathcal{O}'$-reducts of  
$\C$, we have $\V' = \op{HS}(\C')$. The class $\C$ is closed under canonical
extensions (because $\V$ is), and hence so is $\C'$. Canonical
extensions commute with subalgebras 
and homomorphic images (see Lemma~\ref{HSPfin-commut}), so $\V'$ is also closed
under canonical extensions. The result then follows by applying
Theorem~\ref{t:basic} with  the term operation $\boxdot$ defined for the
signature $\mathcal{O}'$.  
\end{proof}

A term $t$ is called {\em positive} if every occurrence of a variable lies within the scope of an even number of occurrences of $\lnot$. Positive terms are expanding (cf.~\cite{Jon94}), so Theorem~\ref{t:BAOs} applies to any $\V$ defined by identities containing positive terms. In particular, it applies to varieties with \emph{conjugate} operators, or more broadly, to all varieties of finite signature, such that $\boxdot$ has a term-definable conjugate, that is, an operator $g$ satisfying
\[
\boxdot x \mt y = 0 \iff x\mt g(y) = 0.
\]
Such varieties are called \emph{cyclic}. For cyclic varieties, the EDPC is
equivalent to being a discriminator variety (see~\cite{KK06}, or~\cite{Kow98}
for the special case of tense algebras), so a stronger result can be stated.

\begin{Corollary}\label{cyclic}
Let $\V$ be a cyclic variety of Boolean algebras with operators that is closed under canonical extensions. If $\V$ is coherent, then  $\V$ is a discriminator variety. 
\end{Corollary}  

Theorem~\ref{t:BAOs} applies of course to varieties of modal algebras. Let $\cls{K}$ be the variety of all modal algebras, and call any subvariety of $\cls{K}$ satisfying the equation $\boxdot^{n+1}x \eq \boxdot^n x$ for some $n \in \N$ \emph{weakly-transitive}.   

\begin{Corollary}
Let\/ $\V$ be any subvariety of\/ $\cls{K}$ that is closed under canonical extensions and not weakly transitive. Then $\V$ is not coherent and does not admit uniform deductive interpolation, and its first-order theory does not have a model completion.
\end{Corollary}

In particular, neither of the varieties corresponding to the well-known modal logics $\lgc{K}$ and $\lgc{KT}$ are coherent, admit uniform deductive interpolation, or have a first-order theory that has a model completion. Note however, that there exist subvarieties of\/ $\cls{K}$, such as the varieties corresponding to $\lgc{K4}$ and $\lgc{S4}$, that are closed under canonical extensions and weakly transitive, and are also not coherent (cf.~\cite{GZ02}).

%%%%%%%%%%%%%%%%%%%%%%%%%%%%%%%%%%%%%%%%%%%%%%%%

\subsection{Varieties of double-Heyting algebras}

A \emph{double-Heyting algebra} is an algebra $\m{A} = \lan A,\mt,\jn,\ra,-,0,1 \ran$ such that $\lan A,\mt,\jn,\ra,0,1\ran$ is a Heyting algebra and $\lan A, \jn,\mt,-,0,1 \ran$ is a dual Heyting algebra, with $-$ dually residuating $\jn$. We consider the unary term $d(x) = (1-x)\ra 0$. This term is decreasing in any double-Heyting algebra $\m{A}$, and $d(\bigwedge X) = \bigwedge\{d(x) \mid x\in X\}$ holds for any $X \subseteq A$ for which $\bigwedge X$ 
is defined, so it also preserves meets of powers.

\begin{Lemma}\label{dbl-H-canon}
The variety of double-Heyting algebras is closed under canonical extensions.
\end{Lemma}

\begin{proof}
Let $\m{A} = \lan A,\jn,\mt,\ra,-,0,1 \ran$ be a double-Heyting algebra, and let $\m{A}^\si = \lan A^\si,\jn^\si,\mt^\si,\ra^\pi,-^\si,0,1 \ran$. Then $\lan A^\si,\jn^\si,\mt^\si,0,1 \ran$ is a complete bounded distributive lattice. It is well known that $\ra^\pi$ residuates $\mt^\si$, so we only need to show that $-^\si$ dually residuates $\jn^\si$. This follows by duality and Lemma~\ref{meet-is-meet}, using the observation that $-$ is an operator when viewed as a map from $A\times A^\partial$ to $A$.
\end{proof}

It was shown in~\cite{Tay16} that a variety of double-Heyting algebras has the EDPC if and only if it is a discriminator variety, and that this situation occurs if and only if the term $d$ defined above is $n$-potent for some $n \in \N$. Hence, we obtain the following  analogue of Corollary~\ref{cyclic}.

\begin{Theorem}\label{t:dbl-H}
Let $\V$ be a variety of double-Heyting algebras that is closed under canonical extensions.  If $\V$ is coherent, then $\V$ is a discriminator variety.
\end{Theorem}  

\begin{proof}
Let $\V$ be a coherent variety of double-Heyting algebras that is closed under canonical extensions. By the observations at the beginning of this subsection, $\V$ satisfies the assumptions of Theorem~\ref{t:basic} with $\boxdot$ as $d$. Hence $\V\models d^{n+1}(x) \eq d^n (x)$ for some $n \in \N$, and the claim follows.
\end{proof}  

Since the variety of all double-Heyting algebras is not a discriminator variety, Lemma~\ref{dbl-H-canon} combined with Theorem~\ref{t:dbl-H} yields the following result.

\begin{Corollary}
The variety of double-Heyting algebras is not coherent and does not admit uniform deductive interpolation, and its first-order theory does not have a model completion.
\end{Corollary}

\noindent
This result may be somewhat surprising in view of the fact that Heyting algebras admit uniform deductive interpolation (cf.~\cite{GZ97}), and bi-intuitionistic logic has the Craig interpolation property (cf.~\cite{KO16}).

%%%%%%%%%%%%%%%%%%%%%%%%%%%%%%%%%%%%%%%%%%%%%%%%

\subsection{Varieties of residuated lattices}

A {\em residuated lattice} is an algebra $\m{A} = \lan A,\mt,\jn,\cdot,\ld,\rd,\e \ran$ such that $\lan A,\mt,\jn
\ran$ is a lattice, $\lan A,\cdot,\e \ran$ is a monoid, and for all $a,b,c \in A$,  
\[
b \le a \ld c \iff a \cdot b \le c \iff a \le c \rd b.
\]
Residuated lattices expanded by a constant $0$ are known as $\FL$-algebras. To present results about $\FL$-algebras and residuated lattices in a uniform way, we will view residuated lattices as $\FL$-algebras satisfying
the identity $\e\eq 0$. We refer to~\cite{GJKO07} for further details regarding these structures and their role as algebraic semantics for substructural logics. 

As remarked already in Section~\ref{s:canonical}, the canonical extension of a residuated lattice must mix $\si$ and $\pi$ extensions of the basic operations. Hence, for a residuated lattice $\m{A}$, its canonical extension is defined to be the algebra $\m{A}^\si = \lan A^\si,\jn^\si,\mt^\si,\cdot^\si,\ld^\pi,\rd^\pi,\e^\si \ran$. Note that the multiplication operation is an operator, and the divisions (with appropriately dualised coordinates) are dual operators. When so defined, $\m{A}^\si$ is a residuated lattice, showing that the variety of residuated lattices is closed under canonical extensions. Many other important varieties of residuated lattices are also closed under canonical extensions as illustrated by the following two lemmas.

\begin{Lemma}\label{can-rls}
Let $\V$ be a variety of $\FL$-algebras defined \textup{(}relative to the variety of all $\FL$-algebras\textup{)} by any combination of the following identities:
\begin{newlist}
\item[$\bullet$] $\e\eq 0$ \textup{(}residuated lattices\textup{)};
\item[$\bullet$] $x\leq \e$ \textup{(}\emph{integral} $\FL$-algebras, or $\FL_o$-algebras\textup{)};
\item[$\bullet$] $0\leq x$ \textup{(}\emph{zero-bounded} $\FL$-algebras, or $\FL_i$-algebras\textup{)};
%\item $0\leq x \leq \e$ \textup{(}integral, zero-bounded $\FL$-algebras, or
%  $\FL_w$-algebras\textup{)},
\item[$\bullet$] $xy \eq yx$ \textup{(}\emph{commutative} $\FL$-algebras, or $\FL_e$-algebras\textup{)};
\item[$\bullet$] $x \leq x^2$ \textup{(}\emph{square-increasing} $\FL$-algebras, or $\FL_c$-algebras\textup{)};
\item[$\bullet$] $0\rd (x\ld 0)\eq (0\rd x)\ld 0$ \textup{(}\emph{cyclic} $\FL$-algebras\textup{)};
\item[$\bullet$] $0\rd (x\ld 0)\eq x \eq (0\rd x)\ld 0$ \textup{(}\emph{involutive} $\FL$-algebras\textup{)};
\item[$\bullet$] $(\e\mt x)^ky \eq y(\e\mt x)^k$, for some $k\in\N$ \textup{(}\emph{Hamiltonian} $\FL$-algebras\textup{)};
\item[$\bullet$] $x\mt(y\jn z) \eq (x\mt y)\jn(x\mt z)$ \textup{(}\emph{distributive} $\FL$-algebras\textup{)}.
  \end{newlist} 
Then $\V$ is closed under canonical extensions. 
\end{Lemma}

\begin{proof}
All these claims are corollaries of the results in Chapter~6 of~\cite{GJKO07}, but to give the reader an idea of how the proofs proceed, we give a proof for the Hamiltonian case here. Suppose that $\m{A} \models (\e\mt x)^ky\eq y(\e\mt x)^k$ for some $k \in \N$. We will show that this identity also holds in $\m{A}^\si$. 

Since multiplication is an operator and $\mt$ is order-preserving, Lemma~\ref{operators} applies (recursively) to the terms $(\e\mt x)^ky$ and $y(\e\mt x)^k$. This yields
\begin{align*}
((\e\mt^\si x)^\si)^k\cdot^\si y
  &= ((\e\mt x)^k)^\si\cdot^\si y)\\    
  &= ((\e\mt x)^k y)^\si\\
  &= (y(\e\mt x)^k)^\si\\
  &= y \cdot^\si((\e\mt^\si x)^\si)^k,
\end{align*}                                   
showing that  $\m{A}^\si\models (\e\mt x)^ky\eq y(\e\mt x)^k$ as required. 
\end{proof}

\begin{Lemma}\label{l:semilinear}
Let $\V$ be a variety of $\FL$-algebras that is closed under canonical extensions. Then also the variety $\V^\ell$ of {\em semilinear} algebras generated by the linearly ordered members of $\V$ is closed under canonical extensions.
\end{Lemma}
\begin{proof}
This result follows directly from Theorem~6.8 in~\cite{GH01}, but let us sketch a proof. Let $\V$ be a variety of $\FL$-algebras that is closed under canonical extensions. Then $\V^\ell = \op{ISP}(\C)$ where $\C$ is the class of all chains (linearly ordered members) of $\V$. Any direct product can be represented as a Boolean product of ultraproducts, so $\V^\ell = \op{ISP}_B\op{P}_U(\C)$, and, since ultraproducts of chains are chains, $\V^\ell = \op{ISP}_B(\C)$. Now, by Lemma~6.7 in~\cite{GH01}, canonical extensions commute with Boolean products, so $(\V^\ell)^\si = \op{ISP}_B(\C)^\si$. But canonical extensions of chains are also chains, so $(\V^\ell)^\si = \op{ISP}_B(\C) = \V^\ell$, as required. 
\end{proof}

It was shown in~\cite{Gal03} that a Hamiltonian variety of residuated lattices $\V$ has the EDPC if and only if $\V\models (\e\mt x)^{n+1}\eq (\e\mt x)^{n}$ for some $n\in\N$. Let $t(x) = (\e\mt x)^2$. Since $t$ is decreasing and expanding in all residuated lattices, an application of Theorem~\ref{t:basic} yields the following result.

\begin{Theorem}\label{t:rls}
Let $\V = \op{ISP}(\C)$ be a coherent variety of residuated lattices such that $\C$ is closed under canonical extensions. Then $\V\models (\e\mt x)^{n+1}\eq (\e\mt x)^{n}$ for some $n \in\N$ and if $\V$ is also Hamiltonian, then $\V$ has the EDPC.
\end{Theorem}  

\begin{Corollary}
Let $\V = \op{ISP}(\C)$ be any variety of residuated lattices such that $\C$ is closed under canonical extensions and $\V \not\models (\e\mt x)^{n+1}\eq (\e\mt x)^{n}$ for all $n \in\N$. Then $\V$ is not coherent and does not admit uniform deductive interpolation, and its first-order theory does not have a model completion.
\end{Corollary}

The lattice-dual form of Theorem~\ref{t:rls} also holds. That is, if $\V$ satisfies the assumptions of the theorem, then $\V\models (\e\jn x)^{m+1}\eq (\e\jn x)^{m}$ for some $m \in\N$. Of the varieties mentioned in Lemma~\ref{can-rls}, those that satisfy both $(\e\mt x)^{n+1}\eq (\e\mt x)^{n}$ for some $n\in\N$, and $(\e\jn x)^{m+1}\eq (\e\jn x)^{m}$ for some $m\in\N$, are term-equivalent to Heyting algebras. Hence, we obtain failures of coherence, uniform deductive interpolation, and existence of a model completion for the first-order theory for varieties of residuated lattices corresponding to all `fundamental' substructural logics, including $\lgc{FL}$, $\lgc{FL_c}$, $\lgc{FL_e}$, $\lgc{FL_w}$, and $\lgc{FL_{ew}}$, and their involutive versions, including $\lgc{MALL}$, the fragment of Linear Logic without exponentials.

%%%%%%%%%%%%%%%%%%%%%%%%%%%%%%%%%%%%%%%%%%%%%%%

\subsection{Varieties of lattices} \label{ss:lattices}

Here we remove the assumption that the signature contains at least one constant. Let $\lang$ be an arbitrary algebraic signature, and let $\V$
be a variety of $\lang$-algebras. The presence or absence of a constant does not affect the definition of coherence, but Proposition~\ref{p:characterization} and hence Theorem~\ref{th:coherent} are not quite correct in this more general setting. Rather than reformulating these results in their entirety, let us just extract the one result that we need. From the proof of Theorem~\ref{th:coherent}, we obtain directly that the following are equivalent:

\begin{newlist}
\item[\rm (1)] $\V$ is coherent.
\item[\rm (2)] For any finite sets $\x,\y$ such that $\F(\y)$ exists and any compact congruence $\The$ on $\F(\x,\y)$, the congruence $\The \cap F(\y)^2$  on $\F(\y)$ is compact.
\end{newlist}

Having removed the requirement that our signature contains constants, let us now allow them to be added. That is, we provide a lemma that allows parameters to be considered as additional constants in an extended signature. Let $\csts$ be a finite non-empty set of constants not in $\lang$. We write $\lang({\csts})$  to denote $\lang$ expanded with $\csts$ and let $\V_{\csts}$ be the variety in this signature defined by all identities that are valid in $\V$. 

\begin{Lemma}\label{l:constants}
If $\V$ is coherent, then so is $\V_{\csts}$.   
\end{Lemma}  
\begin{proof}
Let $\F_{\csts}(\x,\y)$ denote the free algebra of $\V_{\csts}$ over finite disjoint sets $\x,\y$. If we view $\csts$ as a set of variables, then clearly also $F_{\csts}(\x,\y) = F({\csts},\x,\y)$. Now suppose that $\V$ is coherent and let $\De(\x,\y)$ be a finite set of $\lang({\csts})$-equations. Define $\The = \cg{\F_{\csts}(\x,\y)}\De$. Since $\De$ can also be viewed as a set $\De(\csts,\x,\y)$ of $\lang$-equations, $\The = \cg{\F(\csts,\x,\y)}\De$. By the coherence of $\V$, we have that $\The\cap F(\csts,\x)^2$ is compact. But $F(\csts,\x) = F_{\csts}(\x)$, so $\The\cap F_{\csts}(\x)^2$ is also compact, as required.
\end{proof}  

Our final negative result, which makes use of the preceding observations, concerns the variety $\Lat$ of lattices. Let us expand the signature of lattices with the set $\csts = \{c_1,c_2,c_2\}$ of constants, and consider $\Lat_{\csts}$,  whose members we call $\csts$-lattices. Define now 
\[
t(x) = (c_1\mt(c_2\jn(c_3\mt x))) \mt x.
\]
Then $\Lat_{\csts} \models t(x)\leq x$, and $t$ is obviously order-preserving in $\V_{\csts}$. 

\begin{Lemma}\label{c-lattice}
Let $\W$ be a coherent variety of $\csts$-lattices closed under canonical extensions. Then $t$ is $n$-potent for some $n \in \N$. 
\end{Lemma}  
\begin{proof}
Since $t$ is a composition of lattice terms and constants, by Lemma~\ref{isotone}, it is $\W$-expanding. The claim then follows by Theorem~\ref{t:basic}.
\end{proof}

\begin{Corollary}\label{latt-fail}
The variety of lattices is not coherent and does not admit deductive uniform interpolation, and its first-order theory does not have a model completion.
\end{Corollary}
\begin{proof}
Since $\Lat$ is closed under canonical extensions, so is $\Lat_{\csts}$. Moreover, $\Lat_{\csts} \not \models t^{n}(x)\leq t^{n+1}(x)$ for all $n\in\N$. Hence, by Lemma~\ref{c-lattice}, $\Lat_{\csts}$ is not coherent, and, by Lemma~\ref{l:constants}, neither is $\Lat$. 
\end{proof}

The above  negative result for lattices was first proved explicitly by Schmidt in~\cite{Sch83}, although he notes in this paper that an example exhibiting the failure of coherence was known already to McKenzie but unpublished.

%%%%%%%%%%%%%%%%%%%%%%%%%%%%%%%%%%%%%%%%%%%%%%%%

\bibliographystyle{model1a-num-names}

\end{document}